\documentclass[11pt,leqno]{article}
\usepackage{a4wide}
\usepackage{amssymb,upgreek}
\usepackage{amsmath}
\usepackage{amsthm}
\usepackage{mathrsfs}
\usepackage{euscript,csquotes}

\usepackage[curve]{xypic}
\usepackage[usenames,dvipsnames]{xcolor}
\usepackage{endnotes,comment}

\usepackage{hyperref}

\theoremstyle{plain}

\newcommand{\Hom}{\operatorname{Hom}}

\newcommand{\Ext}{\operatorname{Ext}}

\newcommand{\Mod}{\operatorname{Mod}}
\newcommand{\Ind}{\operatorname{Ind}}
\newcommand{\ind}{\operatorname{ind}}

\newcommand{\red}{\operatorname{red}}
\newcommand{\res}{\operatorname{res}}

\newcommand{\chara}{\operatorname{char}}
\newcommand{\iso}{\operatorname{\overset{\sim}{\longrightarrow}}}

\def\U{\EuScript{U}}

\def\F{\mathfrak{F}}
\def\O{\mathfrak{O}}
\def\k{\kappa}
\def\Q{\Bbb{Q}}
\def\N{\Bbb{N}}
\def\Z{\Bbb{Z}}

\def\v{\text{val}_{\mathfrak{F}}}
\def\k{\kappa}
\def\G{\bf{G}}

\def\A{\bf{A}}
\def\S{\bf{S}}
\def\Z{\bf{Z}}
\def\U{\mathcal{U}}

\newtheorem{theorem}{Theorem}[section]
\newtheorem{corollary}[theorem]{Corollary}
\newtheorem{lemma}[theorem]{Lemma}
\newtheorem{proposition}[theorem]{Proposition}

\newtheorem{definition}[theorem]{Definition}

\newtheorem{example}[theorem]{Example}

\theoremstyle{definition}
\newtheorem{remark}[theorem]{Remark}

\title{\textbf{Derived smooth induction with applications}}
\author{Peter Schneider, Claus Sorensen}
\date{}

\begin{document}

\maketitle

\begin{abstract}
In natural characteristic, smooth induction from an open subgroup does not always give an exact functor. In this article we initiate a study of the right derived functors, and we give applications to the non-existence of projective representations and duality.
\end{abstract}

\section{Introduction}

Let $G$ be a profinite group, and let $k$ be a field of characteristic $p$. The category of smooth $G$-representations on $k$-vector spaces $\Mod_k(G)$ has nonzero projective objects if and only if $p$ has finite exponent in the pro-order $|G|$. See \cite[Thm.~3.1]{CK23} for example (or \cite[Rk.~2.20, p.~20]{DK23} for a less precise result).
In this paper we study the question about non-existence of projectives for {\it{locally}} profinite groups. More precisely for $p$-adic Lie groups $G$. We approach the problem via the right derived functors of smooth induction $\Ind_K^G$ from a compact open subgroup $K$. This is the \emph{right} adjoint to the restriction functor and, in contrast to compact induction, the functor $\Ind_K^G$ is provably not exact in general for non-compact $G$.

As a sample result, suppose $G$ is a $p$-adic Lie group with a non-discrete center. We show in Proposition \ref{proj} that the category $\Mod_k(G)$ has no nonzero projective objects.

For general $p$-adic reductive groups (with no restriction on the center) we prove the following result, which has been a folklore expectation for some time:

\begin{theorem}\label{intromain}
Let $G={\bf{G}}(\F)$ for a nontrivial connected reductive group $\G$ defined over a finite extension $\F/\Bbb{Q}_p$. Then $\Mod_k(G)$ has no nonzero projective objects.
\end{theorem}

We deduce \ref{intromain} from the vanishing of $R^d \Ind_K^G(k)$ for certain principal congruence subgroups $K$. More precisely, we fix a special vertex $x_0$ in the Bruhat-Tits building of $G$ and consider the associated group scheme ${\G}_{x_0}^\circ$ over $\O$ (the valuation ring in $\F$). The congruence subgroup
$$
K_m:=\ker\big({\G}_{x_0}^{\circ}(\O) \longrightarrow {\G}_{x_0}^{\circ}(\O/\pi^m\O)\big)
$$
is a uniform pro-$p$ group for $m \in e\N$ (with the extra assumption that $m>e$ if $p=2$). Here $\pi$ is a choice of uniformizer in $\F$, and $e$ denotes the ramification index of
$\F/\Bbb{Q}_p$.

The most technical part of our paper is finding the precise vanishing range for $R^i \Ind_{K_m}^G(k)$. This range is given by the number
$$
i_0:=\dim_{\Bbb{Q}_p}(G/P_{\text{min}})
$$
where $P_{\text{min}}$ denotes the group of $\F$-points of a minimal parabolic $\F$-subgroup of $\G$. We have:

\begin{theorem}\label{introrange}
$R^i \Ind_{K_m}^G(k)=0$ if and only if $i>i_0$.
\end{theorem}

This answers a question in \cite{Sor} about the higher smooth duals of the compactly induced representation $\ind_{K_m}^G(k)$.

We introduce the functor $\underline{\Ind}$ by taking the union of $\Ind_K^G$ as $K$ varies. This takes a smooth $G$-representation to a smooth
$G \times G$-representation. At the derived level this gives a functor
$$
R\underline{\Ind}: D(G) \longrightarrow D(G\times G)
$$
where $D(G):=D(\Mod_k(G))$. In \cite{SS} we studied the smooth duality functor $R\underline{\Hom}(-,k)$ on this category. The complex $R\underline{\Ind}(k)$ in some sense represents
$R\underline{\Hom}(-,k)$ on the subcategory $D(G)^c$ of compact objects. The precise statement is the following: For all compact $V^\bullet$ there is an isomorphism
$$
\tau_{V^\bullet}: R\underline{\Hom}(V^\bullet,k) \overset{\sim}{\longrightarrow} R\underline{\Hom}_{\Mod_k(G_r)}(V^\bullet,R\underline{\Ind}(k))
$$
in $D(G)$. Here, if $V$ is a smooth $G$-representation and $W$ is a smooth $G\times G$-representation, we let $\underline{\Hom}_{\Mod_k(G_r)}(V,W)$ denote the space of $k$-linear maps $V \longrightarrow W$ which are $G$-equivariant on the right and smooth on the left (referring to the two $G$-factors in $G\times G$ acting on $W$).

When $G$ is a $p$-adic reductive group, as in Theorems \ref{intromain} and \ref{introrange}, we show that $R\underline{\Ind}(k)$ only has cohomology in degree zero:

\begin{theorem}\label{introunder}
Keep the group $G={\bf{G}}(\F)$ as in Theorem \ref{intromain}. There is an isomorphism
$$
\mathcal{C}^\infty(G,k)[0] \overset{\sim}{\longrightarrow} R\underline{\Ind}(k)
$$
in $D(G\times G)$. (Here $\mathcal{C}^\infty(G,k)[0]$ is the space of $k$-valued functions on $G$, which are smooth on both sides, viewed as a complex concentrated in degree zero.)
\end{theorem}


\section{The derived functors of smooth induction}

For now $G$ denotes an arbitrary locally profinite group, and $k$ is any field. Let $\Mod_k$ and $\Mod_k(G)$ be the category of $k$-vector spaces and of smooth $G$-representations on $k$-vector spaces, respectively. Fix a compact open subgroup $K\subseteq G$ and consider the restriction functor $\res_K^G: \Mod_k(G) \longrightarrow \Mod_k(K)$. The compact induction functor $\ind_K^G$ is an exact left adjoint.
The full smooth induction functor $\Ind_K^G$ is a right adjoint of $\res_K^G$, but it is not exact in general when the characteristic of $k$ divides the pro-order of $G$. The purpose of this paper is to understand the derived functors $R^i\Ind_K^G$ better in that case.

Starting with an object $V$ from $\Mod_k(K)$ we will follow the convention in \cite[Ch.~I, Sect.~5]{Vig96} and realize $\Ind_K^G(V)$ as the space of all smooth functions
$f:G \longrightarrow V$ satisfying the transformation property $f(\k x)=\k f(x)$ for $\k \in K$. Thus in this article our convention is that $G$ acts by right translations.

\begin{definition}
For an open subgroup $U \subset G$ the $K$-action on $G/U$ gives rise to the following:
\begin{itemize}
\item[i.] Let $K\curvearrowright G/U$ denote the groupoid with objects the elements $x \in G/U$ and morphisms $\Hom(x,y)=\{\k \in K: \k x=y\}$ for $x,y \in G/U$;
\item[ii.] The representation $V$ gives a functor $F_V: K\curvearrowright G/U \longrightarrow \Mod_k$ sending $x \mapsto V^{K \cap xUx^{-1}}$, and if $\k x=y$ the $k$-linear map associated with $\k$ is
\begin{align*}
F_V(\kappa): V^{K \cap xUx^{-1}} & \iso V^{K \cap yUy^{-1}} \\
v &\longmapsto \kappa v.
\end{align*}
\end{itemize}
\end{definition}

We can think of the $U$-invariants $\Ind_K^G(V)^U$ as the limit of $F_V$.

\begin{lemma}\label{Ind}
$\Ind_K^G(V)^U \simeq \varprojlim_{x\in G/U} V^{K \cap xUx^{-1}}$.
\end{lemma}

\begin{proof}
The space $\Ind_K^G(V)^U$ consists of all $K$-equivariant functions $f:G/U \longrightarrow V$. For such an $f$, as $x \in G/U$ varies the vectors
$f(x) \in V^{K \cap xUx^{-1}}$ are compatible via the isomorphisms $F_V(\kappa)$. Vice versa, a compatible tuple of vectors arise from a unique $U$-invariant function.
\end{proof}

\begin{remark}\label{represent}
The point of this categorical description of $\Ind_K^G(V)^U$ is to avoid having to pick double coset representatives $R$ for
$K \backslash G/U$. With such a choice $R$ one can of course describe the $U$-invariants in simpler terms as just a product $\prod_{x\in R} V^{K \cap xUx^{-1}}$. However, as $U$ varies the transition maps become more cumbersome to work with.
\end{remark}

There is a formula for $R^i\Ind_K^G(V)$ of the same nature.

\begin{proposition}\label{RInd}
$R^i\Ind_K^G(V) \simeq \varinjlim_U \varprojlim_{x\in G/U} H^i(K \cap xUx^{-1}, V)$.
\end{proposition}

\begin{proof}
Let $V \longrightarrow J^{\bullet}$ be an injective resolution of $V$ in $\Mod_k(K)$. This remains injective upon restriction to an open subgroup since compact induction is exact.
Therefore we have
\begin{align*}
R^i\Ind_K^G(V) &\simeq h^i(\Ind_K^G(J^{\bullet})) \\
&\simeq {\varinjlim}_U {\varprojlim}_{x\in G/U} h^i((J^{\bullet})^{K \cap xUx^{-1}}) \\
&\simeq  {\varinjlim}_U {\varprojlim}_{x\in G/U} H^i(K \cap xUx^{-1},V).
\end{align*}
In the second isomorphism we moved $h^i$ inside ${\varinjlim}_U$ and ${\varprojlim}_{x\in G/U}$, which is justified by the fact that $\Mod_k$ satisfies
AB5 and AB4* (filtered colimits and products are exact). Recall from Remark \ref{represent} that ${\varprojlim}_{x\in G/U}$ can be identified with a product $\prod_{x\in R}$.
\end{proof}

\begin{remark}\label{trans}
For a fixed $U$ the limit $\varprojlim_{x\in G/U} H^i(K \cap xUx^{-1}, V)$ coincides with the groupoid cohomology of $K\curvearrowright G/U$ as described in \cite[Df.~6]{Ron} for example. It is the limit of the functor $F_V^i$ sending $x \mapsto H^i(K \cap xUx^{-1}, V)$. Concretely, an element of this limit is a function
$$
c: G/U \longrightarrow {\bigoplus}_{x\in G/U} H^i(K \cap xUx^{-1}, V)
$$
with the following properties:
\begin{itemize}
\item[i.] $c_x:=c(x) \in H^i(K \cap xUx^{-1}, V)$ for all $x \in G/U$;
\item[ii.] If $\k x=y$ then $c_x \mapsto c_y$ via the isomorphism
$$
\k_*=F_V^i(\k): H^i(K \cap xUx^{-1},V) \iso H^i(K \cap yUy^{-1},V).
$$
\end{itemize}
With this description we can make the transition maps in the colimit $\varinjlim_U$ explicit. Let $U' \subset U$ be an open subgroup. Then the transition map is
\begin{align*}
t_{U,U'}^i: {\varprojlim}_{x\in G/U} H^i(K \cap xUx^{-1}, V)  & \longrightarrow  {\varprojlim}_{x'\in G/U'} H^i(K \cap x'U'x'^{-1}, V) \\
c & \longmapsto \big(\res_{K \cap x'U'x'^{-1}}^{K \cap x'Ux'^{-1}}c_{x'U}\big)_{x'\in G/U'}.
\end{align*}
This above formula for $t_{U,U'}^i$ will play a crucial role throughout this paper.
\end{remark}


\section{The connection to higher smooth duality}

To motivate the ensuing discussion we establish a relation between the functors $R^{i}\Ind_K^G$ and the higher smooth duality functors introduced in \cite{Koh}, and recast in \cite{SS}.

\begin{lemma}\label{Ind-ind}
For $V$ in $\Mod_k(K)$ and $W$ in $\Mod_k(G)$ there are functorial isomorphisms
$$
\Ind_K^G \underline{\Hom}(V,W|_K) \simeq  \underline{\Hom}(\ind_K^G V,W).
$$
(Here $\underline{\Hom}$ denotes the smooth $k$-linear maps, as defined in \cite[Sect.~1]{SS} for example.)
\end{lemma}

\begin{proof}
For any representation $X$ in $\Mod_k(G)$ we have functorial isomorphisms
\begin{align*}
\Hom_{\Mod_k(G)}(X,\Ind_K^G \underline{\Hom}(V,W|_K)) &\simeq \Hom_{\Mod_k(K)}(X|_K, \underline{\Hom}(V,W|_K)) \\
&\simeq \Hom_{\Mod_k(K)}(X|_K \otimes_k V, W|_K) \\
&\simeq  \Hom_{\Mod_k(G)}(\ind_K^G(X|_K \otimes_k V), W) \\
&\simeq \Hom_{\Mod_k(G)}(X \otimes_k \ind_K^G V, W) \\
&\simeq \Hom_{\Mod_k(G)}(X, \underline{\Hom}(\ind_K^G V,W)).
\end{align*}
The fourth isomorphism follows from \cite[p.~40]{Vig96}; part d) just prior to Section 5.3. The others use standard adjunction properties.
The claim then follows from the Yoneda lemma.
\end{proof}

This gives the following spectral sequence (with $V$ and $W$ as above).

\begin{proposition}\label{spec}
$E_2^{i,j}=R^i\Ind_K^G \underline{\Ext}^j(V,W|_K) \Longrightarrow \underline{\Ext}^{i+j}(\ind_K^G V,W)$.
\end{proposition}

\begin{proof}
Note that the functor $\underline{\Hom}(V,-)$ preserves injective objects since $(-)\otimes_kV$ is an exact left adjoint. So does $(-)|_K$ as observed in the proof of Proposition \ref{RInd}.
The Grothendieck spectral sequence for $\Ind_K^G$ composed with $\underline{\Hom}(V,(-)|_K)$ takes the stated form by \ref{Ind-ind}.
\end{proof}

We emphasize the special case $W=k$ below.

\begin{corollary}\label{degen}
Suppose $V$ is a finite-dimensional object of $\Mod_k(K)$ and let $V^*=\Hom_k(V,k)$ denote its contragredient. Then there is an isomorphism of $G$-representations
$$
R^i\Ind_K^G (V^*) \simeq \underline{\Ext}^i(\ind_K^G V,k).
$$
\end{corollary}

\begin{proof}
When $W=k$ the spectral sequence in
Proposition \ref{spec} becomes
$$
E_2^{i,j}=R^i\Ind_K^G \underline{\Ext}^j(V,k) \Longrightarrow \underline{\Ext}^{i+j}(\ind_K^G V,k).
$$
When $V$ is finite-dimensional $\underline{\Hom}(V,-)=\Hom_k(V,-)$ is exact, so $\underline{\Ext}^j(V,k)=0$ for $j>0$ and
$\underline{\Ext}^0(V,k)=V^*$. In this case the spectral sequence degenerates into the isomorphisms in \ref{degen}, and we are done.
\end{proof}


\section{The top-dimensional derived functor}\label{top}

In this section we take $G$ to be a $p$-adic Lie group of dimension $d=\dim_{\Q_p}(G)$, and we assume $\chara(k)=p$.

\begin{remark}\label{rem:top}
In part i of Proposition \ref{prop:Ind-fcd} we will show that, for any $V$ in $\Mod_k(K)$,
$$
R^i\Ind_K^G(V) = 0 \: \: \: \: \forall i>d.
$$ 
\end{remark}

A natural question is whether it is possible to compute the top-dimensional derived functors $R^d\Ind_K^G(V)$. For the trivial representation $V=k$ we have the following.

\begin{proposition}\label{proj}
Assume $G$ has a non-discrete center. Then:
\begin{itemize}
\item[i.] $R^d\Ind_K^G(k)=0$ for all compact open subgroups $K \subset G$;
\item[ii.] The category $\Mod_k(G)$ has no nonzero projective objects.
\end{itemize}
\end{proposition}

\begin{proof}
For the proof of part one let $U \subset K$ be any open Poincar\'{e} subgroup, and let $c$ be as in Remark \ref{trans} with $V=k$. We must find an open subgroup
$U' \subset U$ such that $t_{U,U'}^d(c)=0$. The corestriction map
$$
\text{cor}_{K \cap x'U'x'^{-1}}^{K \cap x'Ux'^{-1}}: H^d(K \cap x'U'x'^{-1},k) \longrightarrow H^d(K \cap x'Ux'^{-1},k)
$$
is known to be an isomorphism (of one-dimensional spaces) for all $U'$. Its composition with the restriction map $\res_{K \cap x'U'x'^{-1}}^{K \cap x'Ux'^{-1}}$ is multiplication by the index. Thus $\res_{K \cap x'U'x'^{-1}}^{K \cap x'Ux'^{-1}}=0$ if this index is $>1$. To summarize, we must find $U'\subsetneq U$ such that we have strict inclusions
$$
K \cap gU'g^{-1} \subsetneq K \cap gUg^{-1}
$$
for all $g \in G$. Let $Z$ denote the center of $G$. Intersecting both sides above with $Z$ shows it is enough to pick a $U'$ such that $Z \cap U' \subsetneq Z\cap U$. For example,
for the right-hand side we get
$$
Z \cap (K \cap gUg^{-1})= K \cap g(Z \cap U)g^{-1}=K \cap(Z \cap U)=Z \cap U.
$$
Since we are assuming $Z$ is non-discrete $Z \cap U$ contains a non-identity element $z$ say. Pick an open neighborhood
$z(Z \cap U') \subset Z \cap U$ not containing the identity. Any such $U'$ works.

For part two let $W$ be an arbitrary object of $\Mod_k(G)$ and let $V$ be an object of $\Mod_k(K)$ as before. Note that $\Ind_{K}^G$ takes injective objects to injective objects
(since restriction is an exact left adjoint) and we therefore have a Grothendieck spectral sequence of the form
$$
E_2^{j,i}=\Ext_{\Mod_k(G)}^j(W, R^i\Ind_{K}^G(V)) \Longrightarrow \Ext_{\Mod_k(K)}^{i+j}(W,V)
$$
coming from Frobenius reciprocity. See \cite[p.~42]{Vig96} for example. If $\Hom_{\Mod_k(G)}(W,-)$ is exact we find that $E_2^{j,i}=0$ for all $j>0$, from which we deduce an isomorphism
$$
\Hom_{\Mod_k(G)}(W, R^i\Ind_{K}^G(V)) \simeq  \Ext_{\Mod_k(K)}^{i}(W,V).
$$
We specialize to the case $V=k$ and $i=d$. As we have just shown, the left-hand side vanishes in this case.
For part two $K$ only plays an auxiliary role and we may take it to be Poincar\'{e}. We infer that
$$
\Hom_k(H^0(K,W),k)\simeq \Ext_{\Mod_k(K)}^{d}(W,k)=0 \Longrightarrow H^0(K,W)=0 \Longrightarrow W=0
$$
by duality for Poincar\'{e} groups. See the review in \cite[Sect.~1]{SS} for instance.
\end{proof}

\begin{remark}\label{discrete}
Both parts of the previous Proposition clearly fail if $G$ is discrete (as $d=0$ and smooth $G$-representations are the same as abstract $k[G]$-modules). We do not know whether
the Proposition holds if we only assume $G$ itself is non-discrete.

Also, a natural question in the context of Proposition \ref{proj} is whether every homotopically projective complex of objects in $\Mod_k(G)$ is necessarily contractible. An affirmative answer would vastly generalize part ii of Proposition \ref{proj}.
\end{remark}

Theorem \ref{intromain} in the introduction gives a supplement to Proposition \ref{proj} for $p$-adic reductive groups $G$.


\section{The case of $p$-adic reductive groups}\label{red}

\subsection{Notation, conventions, and background}

In this article $\N=\{1,2,3,\ldots\}$ denotes the set of all positive integers.

We let $\F/\Q_p$ be a finite extension with valuation ring $\O$, and we choose a uniformizer $\pi$.
Take $\v$ to be the valuation on $\F$ satisfying $\v(\pi)=1$. As usual $q=p^f$ is the cardinality of the residue field $\O/\pi\O$, and $e=e(\F/\Q_p)$ denotes the ramification index.

In this section $\G$ is an arbitrary connected reductive group defined over $\F$, and $G=\G(\F)$. We choose a maximal $\F$-split subtorus ${\S} \subset {\G}$ and let $\Z$ denote its centralizer. Following our earlier convention,
$S=\S(\F)$ and $Z=\Z(\F)$ denote the $p$-adic Lie groups of $\F$-rational points.

Let $\Phi$ denote the roots of $\G$ relative to $\S$, and select a subset of positive roots $\Phi^+$ once and for all (then $\Phi^-=-\Phi^+$ is the set of negative roots). The root system may not be reduced, and we let $\Phi_{\red}$ be the subset of reduced roots ($\alpha \in \Phi$ such that $\frac{1}{2}\alpha \notin \Phi$). By $\Phi_{\red}^{+}$ and $\Phi_{\red}^{-}$ we mean the subsets of positive and negative reduced roots respectively.

Furthermore, we pick a special vertex $x_0$ in the apartment associated with $\S$, and consider the Bruhat-Tits group scheme ${\G}_{x_0}^\circ$ over $\O$. The special subgroup
$K_0={\G}_{x_0}^\circ(\O)$ and its principal congruence subgroups
$$
K_m=\ker\big({\G}_{x_0}^{\circ}(\O) \longrightarrow {\G}_{x_0}^{\circ}(\O/\pi^m\O)\big)
$$
play a pivotal role. The argument proving \cite[Cor.~7.8]{OS19} applies verbatim and shows $K_m$ is a uniform pro-$p$ group if $m \in e\N$ and $m>e$ if $p=2$. We give the proof below and compute the lower $p$-series for such $K_m$.

Let $\A$ be the maximal $\F$-split subtorus of the center of $\G$. We can arrange for $x_0$ to be the origin of the apartment $X_*({\S})/X_*({\A})\otimes \Bbb{R}$. In this case the
action of $z \in Z$ on the apartment is translation by the image of $\nu(z)$ where
$$
\nu: Z \longrightarrow X_*({\S})\otimes \Bbb{R}
$$
is the homomorphism for which
$$
\langle \nu(z),\chi | {\S} \rangle = -\v \chi(z)
$$
for all $\chi \in X^*({\Z})$. We refer to \cite[Sect.~I.1]{SSt} for more details. Later on we consider the monoid $Z^+$ of all $z \in Z$ such that $\langle \nu(z),\alpha \rangle \leq 0$ for all $\alpha \in \Phi^+$.

For each $\alpha \in \Phi$ we have a root subgroup $\U_\alpha$ of $\G$ with $\F$-rational points $U_\alpha$; the latter is normalized by $Z$. According to \cite[Thm.~9.6.5]{KP23} the root datum of $\G$ carries a valuation attached to $x_0$. By \cite[Df.~6.1.2]{KP23} this gives rise to a descending filtration of $U_\alpha$ by subgroups $(U_{\alpha,r})_{r\in \Bbb{R}}$ satisfying
\begin{equation}\label{f:ZconjU}
  zU_{\alpha,r}z^{-1} = U_{\alpha,r - \langle \nu(z),\alpha \rangle}  \qquad\text{for any $z \in Z$}.  \footnote{The $\nu$ in \cite[Df.~6.1.2 V6]{KP23} is the composite of our $\nu$ and the canonical splitting of the natural monomorphism $X_*({\S}') \otimes \Bbb{R} \hookrightarrow X_*({\S}) \otimes \Bbb{R}$ where ${\S}'$ is the maximal subtorus of ${\S}$ contained in the derived subgroup of $\G$ (\cite[4.1.4]{KP23}). This canonical splitting exists since the difference between ${\S}$ and ${\S}'$, up to isogeny, comes from the center of $\G$. Hence our number $\langle \nu(z),\alpha \rangle$, for a root $\alpha \in \Phi$, coincides with $\alpha(\nu(z))$ in \cite[Df.~6.1.2 V6]{KP23}. Note that \cite{KP23} contains a sign mistake as explained in \cite[2nd item on p.\ 23]{Kal}.
}
\end{equation}
This filtration needs to be modified. For any concave function $f$ on $\Phi \cup \{0\}$ \cite[Df.~7.3.3]{KP23} introduces the subgroup
\begin{equation*}
  U_{\alpha,f} := U_{\alpha,x_0,f} := U_{\alpha,f(\alpha)} \cdot U_{2\alpha,f(2\alpha)}
\end{equation*}
of $U_\alpha$. For any real number $r \geq 0$ the constant function $f_r$ with value $r$ is concave, and we abbreviate $\tilde{U}_{\alpha,r} := U_{\alpha,f_r}$.\footnote{\cite[Sect.~13.2]{KP23} replaces $f_r$ again by $r$ in the notation, which we find too confusing. For simplicity we drop the point $x_0$, which we fixed once and for all, from the notation.}

Before \cite[Lem.~9.8.1]{KP23} a descending filtration $(Z_r)_{r \geq 0}$ of $Z$ is constructed. We point out that this filtration depends on the connected reductive group $\Z$ and not the ambient group $\G$. Since the Bruhat-Tits building of $Z$ is a single point (\cite[Prop.~9.3.9]{KP23}), namely $x_0$, it follows from \cite[Prop.~13.2.5, part (2)]{KP23} that each $Z_r$ is a normal subgroup of $Z$.

Using these filtrations \cite[Def.~7.3.3]{KP23} then introduces a descending filtration $(\mathcal{P}_{x_0,r})_{r \geq 0}$ of $K_0 = {\G}_{x_0}^\circ(\O)$. The crucial property of these filtrations is the Iwahori factorization (\cite[Prop.~13.2.5, part (3)]{KP23}): For any $r > 0$ the multiplication map defines a homeomorphism
\begin{equation}\label{f:Iwahori-fac}
\prod_{\alpha \in \Phi_{\red}^-} \tilde{U}_{\alpha,r} \times Z_r \times \prod_{\alpha \in \Phi_{\red}^+} \tilde{U}_{\alpha,r} \iso \mathcal{P}_{x_0,r} \ .
\end{equation}
(The factors in each product are ordered in a fixed but arbitrary way.)

\begin{lemma}\label{lem:congruence-subgroup}
   For any $m > 0$ we have $K_m = \mathcal{P}_{x_0,m}$.
\end{lemma}
\begin{proof}
By \cite[Prop.~9.8.3]{KP23} there is, for any $r \geq 0$, a smooth affine $\O$-group scheme of finite type  $\mathcal{G}_{x_0,r}$ such that $\mathcal{G}_{x_0,r}(\O)=\mathcal{P}_{x_0,r}$. It comes by descend from the maximal unramified extension $\F^\text{ur}$ of $\F$. By \cite[Prop.~8.5.16, Df.~A.5.12]{KP23} and passing to $\text{Gal}(\F^\text{ur}/\F)$-invariants we have
$$
\mathcal{G}_{x_0,r+m}(\O)=\ker\big(\mathcal{G}_{x_0,r}(\O) \longrightarrow \mathcal{G}_{x_0,r}(\O/\pi^m\O)\big).
$$
Now take $r=0$ and observe that $\mathcal{G}_{x_0,0}={\G}_{x_0}^\circ$ since $K_0=\mathcal{P}_{x_0,0}$.
\end{proof}

As a consequence we obtain from \eqref{f:Iwahori-fac} the Iwahori factorization
\begin{equation}\label{f:Iwahori-fac-Km}
\prod_{\alpha \in \Phi_{\red}^-} \tilde{U}_{\alpha,m} \times Z_m \times \prod_{\alpha \in \Phi_{\red}^+} \tilde{U}_{\alpha,m} \iso K_m    \qquad\text{for any $m >0$}.
\end{equation}

\begin{remark}\label{rem:factors-pcs}
For any $m > 0$ and $\alpha \in \Phi_{\red}$ the filtration subgroups $\tilde{U}_{\alpha,m}$ and $Z_m$ are themselves principal congruence subgroups. For $Z_m$ this is a special case of Lemma \ref{lem:congruence-subgroup} since $Z_m$ is the analog of $\mathcal{P}_{x_0,m}$ for the connected reductive $\F$-group $\mathbf{Z}$.

To see that  $\tilde{U}_{\alpha,m}$ is a principal congruence subgroup we argue as follows. By \cite[Prop.~2.11.9, Prop.~9.8.3]{KP23} there is, for $r \geq 0$, a unique closed subgroup scheme
${\U}_{\alpha,x_0,r} \subset \mathcal{G}_{x_0,r}$ such that ${\U}_{\alpha,x_0,r}(\O) = \tilde{U}_{\alpha,r}$. Because ${\U}_{\alpha,x_0,r}$ is closed the map
${\U}_{\alpha,x_0,r}(\O/\pi^m\O) \hookrightarrow  \mathcal{G}_{x_0,r}(\O/\pi^m\O)$ is still an inclusion and therefore
\begin{align*}
\tilde{U}_{\alpha,m} &= \mathcal{P}_{x_0,m} \cap \tilde{U}_{\alpha,0} = K_m \cap \tilde{U}_{\alpha,0}  \\
&= \ker\big({\G}_{x_0}^{\circ}(\O) \longrightarrow {\G}_{x_0}^{\circ}(\O/\pi^m\O)\big) \cap \tilde{U}_{\alpha,0}\\
&=\ker\big({\U}_{\alpha,x_0,0}(\O) \longrightarrow {\G}_{x_0}^{\circ}(\O/\pi^m\O)\big) \\
&=\ker\big({\U}_{\alpha,x_0,0}(\O) \longrightarrow {\U}_{\alpha,x_0,0}(\O/\pi^m\O)\big) .
\end{align*}
The first equality follows from \cite[Prop.~13.2.5, part (3)]{KP23}, applied to $\mathcal{P}_{x_0,m}$, after comparing the components associated with the root $\alpha$. The second equality is Lemma \ref{lem:congruence-subgroup}.
\end{remark}

\subsection{Some uniform subgroups and their cohomology}

We begin with the following variant of \cite[Cor.~7.7]{OS19}. Let $\mathcal{G}=\text{Spec}(A)$ be a smooth affine $\O$-group scheme, and $\widehat{\mathcal{G}}=\text{Spf}(\widehat{A}^\frak{p})$ be its formal completion in the unit section. Upon choosing an isomorphism $\widehat{A}^\frak{p}\simeq \frak{O}[\![X_1,\ldots,X_\delta ]\!]$ we get a homeomorphism
$\xi: \widehat{\mathcal{G}}(\frak{O}) \overset{\sim}{\longrightarrow} (\pi\O)^\delta$. We let
$$
\widehat{\mathcal{G}}_m(\frak{O}):=\xi^{-1}\big((\pi^m\frak{O})^\delta\big)=\ker\big(\mathcal{G}(\frak{O})\longrightarrow \mathcal{G}(\frak{O}/\pi^m\frak{O})\big).
$$
With this notation we have the following slight generalization of \cite[Cor.~7.7]{OS19}.

\begin{lemma}\label{variant}
Let $m\in e\N$, and assume $m>e$ if $p=2$. Then the congruence subgroup $\widehat{\mathcal{G}}_m(\frak{O})$ is a uniform pro-$p$ group. Furthermore
$$
\widehat{\mathcal{G}}_{m+e}(\frak{O})=\widehat{\mathcal{G}}_m(\frak{O})^p=\{g^p: g \in \widehat{\mathcal{G}}_m(\frak{O})\}.
$$
(This gives the lower $p$-series for $\widehat{\mathcal{G}}_m(\frak{O})$ by iteration.)
\end{lemma}

\begin{proof}
The fact that $\widehat{\mathcal{G}}_m(\frak{O})$ is uniform is precisely the content of \cite[Cor.~7.7]{OS19}. Here we compute its lower $p$-series. By considering $\mathcal{G}_0=\text{Res}_{\frak{O}/\Bbb{Z}_p}\mathcal{G}$, and noting that $\widehat{\mathcal{G}}_{0,j}(\Bbb{Z}_p)=\widehat{\mathcal{G}}_m(\frak{O})$ if $m=ej$, we may (and will) assume $\frak{O}=\Bbb{Z}_p$.

We first deal with the case $p>2$. As explained in \cite[7.2.1]{OS19} the group  $\widehat{\mathcal{G}}(\Bbb{Z}_p)$ is a standard group in the sense of \cite[Df.~8.22]{DdSMS}. The (last paragraph of the) proof of \cite[Thm.~8.31]{DdSMS} shows that $\widehat{\mathcal{G}}(\Bbb{Z}_p)^{p^{m-1}}=\widehat{\mathcal{G}}_m(\Bbb{Z}_p)$. (See also part (iii) of \cite[Thm.~3.6]{DdSMS}
which gives the lower $p$-series of a uniform group.) Raising both sides to the $p^\text{th}$ power gives the result.

For $p=2$ we work with $\widehat{\mathcal{G}}_2(\Bbb{Z}_2)$ which again is standard for the same reason. The proof of \cite[Thm.~8.31]{DdSMS} (with $\varepsilon=1$) also shows that $\widehat{\mathcal{G}}_2(\Bbb{Z}_2)^{2^{m-2}}=\widehat{\mathcal{G}}_{m}(\Bbb{Z}_2)$ for $m \geq 2$. Taking squares gives the result.
\end{proof}

In particular $K_m$ as well as $\tilde{U}_{\alpha,m}$ and $Z_m$ (by Remark \ref{rem:factors-pcs}) all are uniform pro-$p$ groups if $m \in e\N$ and $m>e$ if $p=2$; moreover
\begin{equation}\label{f:lower-series}
  K_m^p = K_{m+e} , \ Z_m^p = Z_{m+e}, \ \text{and}\quad  \tilde{U}_{\alpha,m}^p = \tilde{U}_{\alpha,m+e} \ .
\end{equation}

Next we will see that in certain situations intersections of uniform subgroups are uniform. In the following we put $\wp := p$ if $p>2$, and $\wp := 4$ if $p=2$.

\begin{lemma}\label{lem:intersec-uniform}
  Let $(A,\|\cdot\|)$ be a finite dimensional normed $\mathbb{Q}_p$-algebra whose norm $\|\cdot\|$ is submultiplicative and let $A_0 := \{ a \in A : \|a\| \leq \wp\}$. Fix a closed subgroup $\Gamma \leq 1+A_0$. Let $K,K'$ be uniform open subgroups of $\Gamma$, and $a \in A^{\times}$ an arbitrary unit. If $a^{-1}K'a\cap K$ is open in $K$ then it is uniform.
\end{lemma}
\begin{proof}
First of all note that $A_0$ is a open pro-$p$ subgroup of the additive group $A$, and $1 + A_0$ is an open pro-$p$ subgroup of the multiplicative group $A^\times$. It is well known (cf.\ \cite[Prop.\ 6.22 and Cor.\ 6.25]{DdSMS}) that the usual exponential and logarithm power series induce homeomorphisms
\begin{equation*}
  \exp: A_0 \longrightarrow 1+A_0 \quad\text{and}\quad \log: 1+A_0 \longrightarrow A_0 \ ,
\end{equation*}
which are inverse to each other.

We recall that a pro-$p$ group $H$ is uniform if and only if the following conditions are satisfied (\cite[Thm.\ 4.5]{DdSMS}):
\begin{itemize}
\item[(1)] $H$ is (topologically) finitely generated;
\item[(2)] $H$ is torsion-free;
\item[(3)] $H$ is powerful -- which means $H/\overline{H^\wp}$ is abelian.
\end{itemize}
Here $H^n$, for $n \in \mathbb{N}$, is a priori the subgroup generated by the set of $n$-powers. However, when $H$ is uniform $H^\wp$ is the same as the set of all $\wp$-powers by \cite[Lemma 3.4]{DdSMS} ; the same result shows that $H^\wp$ is open.

Obviously the intersection is torsion-free since $K$ is. Since $a^{-1}K'a\cap K$ is open in $K$ by assumption it is finitely generated by \cite[Prop.\ 1.7]{DdSMS} .
It remains to show it is powerful, so choose two elements $x,y \in a^{-1}K'a \cap K$ arbitrarily. It suffices to show the commutator
$[x,y]$ is a $\wp$-power from $a^{-1}K'a\cap K$. Since $K$ and $a^{-1}K'a \cong K'$ are uniform we can write
\begin{equation*}
  [x,y] = \kappa^\wp = (a^{-1} \kappa' a)^\wp
\end{equation*}
for suitable $\kappa \in K$ and $\kappa' \in K'$. We just need to argue that $\kappa = a^{-1} \kappa' a$. Taking $\log$ above, using \cite[Cor.\ 6.25(iii)]{DdSMS}, yields
$\log[x,y] = \wp \log \kappa$. Since $a^{-1} \kappa' a$ may not lie in $1+A_0$ we have to note that nevertheless $\log(a^{-1} \kappa' a)$ makes sense. Indeed the series
\begin{equation*}
  \log(a^{-1} \kappa' a) = \sum_{n=1}^{\infty} \frac{(-1)^{n+1}}{n}(a^{-1} \kappa' a-1)^n = \sum_{n=1}^{\infty} \frac{(-1)^{n+1}}{n}a^{-1}(\kappa' -1)^n a = a^{-1} \log(\kappa') a
\end{equation*}
still converges. Having checked this, taking $\log$ as above we deduce $\log[x,y] = \wp \log(a^{-1} \kappa' a)$. Comparing the two expressions for $\log[x,y] \in \wp A_0$ and dividing by $\wp$ (in $A$) we infer the equality $\log(\kappa) = \log(a^{-1} \kappa' a)$ in $A_0$. Finally take $\exp$ on both sides, again noting that
\begin{equation*}
  \exp\big(\log(a^{-1} \kappa' a) \big) = \exp \big(a^{-1}\log(\kappa')a\big) = a^{-1}\exp \big(\log(\kappa')\big)a = a^{-1}\kappa'a \ .
\end{equation*}
This gives the equality $\kappa = a^{-1} \kappa' a$ as desired. We conclude $a^{-1} K' a \cap K$ that is uniform.
\end{proof}

\begin{example}
Our main example is the matrix algebra $A = M_N(\mathbb{Q}_p)$ with the norm $\|a\|=\max_{i,j} |a_{ij}|$. The multiplicative group $1+A_0$ is the congruence subgroup $\ker(\mathrm{GL}_N(\mathbb{Z}_p)\rightarrow \mathrm{GL}_N(\mathbb{Z}_p/\wp\mathbb{Z}_p))$. Then  Lemma \ref{lem:intersec-uniform} implies the following: If $K$ and $K'$ are uniform closed subgroups of $\ker(\mathrm{GL}_N(\mathbb{Z}_p)\rightarrow \mathrm{GL}_N(\mathbb{Z}_p/\wp \mathbb{Z}_p))$ then so is $a^{-1} K' a \cap K$ for any $a \in \mathrm{GL}_N(\mathbb{Q}_p)$ such that $a^{-1} K' a \cap K$ is open in $K$.
\end{example}

\begin{proposition}\label{prop:intersec-uniform}
Let $m,n \in e\N$ assuming $m,n>e$ if $p=2$; we then have:
\begin{itemize}
  \item[a.] $K_m \cap gK_ng^{-1}$ is uniform for all $g \in G$;
  \item[b.] $\tilde{U}_{\alpha,m} \cap z \tilde{U}_{\alpha,n} z^{-1}$ is uniform for all $\alpha \in \Phi$ and $z \in Z$,
\end{itemize}
\end{proposition}
\begin{proof}
As we noted before \eqref{f:lower-series} all groups of which we take intersections are uniform. The assertion then follows from the above Example by using a faithful representation of the Weil restriction to $\mathbb{Z}_p$ of $\mathbf{G}_{x_0}^\circ$ into some $\mathrm{GL}_N$.
\end{proof}

Throughout we fix a field $k$ of characteristic $p$. A uniform pro-$p$ group $U$ admits a canonical $p$-valuation defined by the lower $p$-series. Thus $U$ becomes equi-$p$-valued, and by \cite[Ch.~V, Prop.~2.5.7.1]{Laz} the cup product gives an isomorphism of graded $k$-algebras
$$
\bigwedge H^1(U,k) \iso H^*(U,k).
$$
Moreover $H^1(U,k)=\Hom_{\Bbb{F}_p}(U/U^p,k)$ is the dual Frattini quotient, which we will henceforth denote $U_{\Phi}^*$. (Here the subscript $\Phi$ is standard notation for Frattini quotients, and should not be confused with the root system.) Our exterior products $\bigwedge=\bigwedge_{k}$ are always over $k$.

\begin{proposition}\label{decompose}
Consider arbitrary $m,n \in e\N$, both assumed to be $>e$ if $p=2$, and $z \in Z$. Then $H^i(K_m \cap zK_nz^{-1},k)$ decomposes as a direct sum
$$
\bigoplus \bigg[\bigotimes_{\alpha \in \Phi_{\red}^-} \bigwedge^{a_{\alpha}} (\tilde{U}_{\alpha,m}\cap z\tilde{U}_{\alpha,n}z^{-1})_{\Phi}^* \otimes
\bigwedge^b (Z_{\max\{m,n\}})_{\Phi}^* \otimes \bigotimes_{\alpha \in \Phi_{\red}^+} \bigwedge^{c_{\alpha}} (\tilde{U}_{\alpha,m}\cap z\tilde{U}_{\alpha,n}z^{-1})_{\Phi}^*\bigg]
$$
where $\{a_{\alpha}\}_{\alpha\in \Phi_{\red}^-}$, $b$, $\{c_{\alpha}\}_{\alpha\in \Phi_{\red}^+}$ in the sum $\bigoplus$ range over non-negative integers with sum $i$.
\end{proposition}
\begin{proof}
Conjugating the Iwahori factorization of $K_n$ by $z$, using that $Z_n$ is normal in $Z$, and intersecting with $K_m$ we find that
\begin{equation}\label{iwa}
\big(\prod_{\alpha \in \Phi_{\red}^-} \tilde{U}_{\alpha,m}\cap z\tilde{U}_{\alpha,n}z^{-1}\big)  \times Z_{\max\{m,n\}} \times \big(\prod_{\alpha \in \Phi_{\red}^+} \tilde{U}_{\alpha,m}\cap z\tilde{U}_{\alpha,n}z^{-1} \big) \iso K_m \cap zK_nz^{-1}.
\end{equation}
By the same argument as in the proof of \cite[Cor.~7.11]{OS19}, which uses the uniformity of all groups involved, quotienting out $p$-powers induces an isomorphism on the level of Frattini quotients. Taking $k$-linear duals and forming the exterior algebra
gives an isomorphism of graded $k$-algebras
\begin{align*}
   & \bigwedge (K_m \cap zK_nz^{-1})_{\Phi}^* \xrightarrow{\ \simeq\ } \\
   & \bigwedge\bigg[ \bigoplus_{\alpha \in \Phi_{\red}^-} (\tilde{U}_{\alpha,m}\cap z\tilde{U}_{\alpha,n}z^{-1})_{\Phi}^* \oplus (Z_{\max\{m,n\}})_{\Phi}^* \oplus \bigoplus_{\alpha \in \Phi_{\red}^+} (\tilde{U}_{\alpha,m}\cap z\tilde{U}_{\alpha,n}z^{-1})_{\Phi}^*\bigg] \simeq  \\
   & \bigotimes_{\alpha \in \Phi_{\red}^-} \bigwedge (\tilde{U}_{\alpha,m}\cap z\tilde{U}_{\alpha,n}z^{-1})_{\Phi}^* \otimes
\bigwedge (Z_{\max\{m,n\}})_{\Phi}^* \otimes \bigotimes_{\alpha \in \Phi_{\red}^+} \bigwedge (\tilde{U}_{\alpha,m}\cap z\tilde{U}_{\alpha,n}z^{-1})_{\Phi}^* \ .
\end{align*}
Comparing the degree $i$ graded pieces gives the result.
\end{proof}

\subsection{On the vanishing of certain restriction maps}

We fix an $m \in e\N$. If $p=2$ we always assume $m>e$. For $n,n'\in e\N$ such that $n\leq n'$ we have restriction maps
$$
\res_{n,n'}^i(g): H^i(K_m \cap gK_ng^{-1},k) \longrightarrow H^i(K_m \cap gK_{n'}g^{-1},k).
$$
When $g=z\in Z$, this map is compatible with the decomposition in Prop.\ \ref{decompose} in the obvious sense. (In general restriction commutes with cup products in group cohomology.)

\begin{lemma}\label{resvan}
Suppose $m\leq n<n'$ all lie in $e\N$ (and are $>e$ if $p=2$). Then $\res_{n,n'}^i(g)=0$ for all $g \in G$ and
$i>i_0:=\dim_{\Bbb{Q}_p}(G/P_{\text{min}})$.
\end{lemma}

\begin{proof}
By the Cartan decomposition $G=K_0Z^+K_0$, as described in \cite[Thm.~5.2.1, part (1)]{KP23} for example,
we may write $g=hzh'$ for some $z \in Z^+$ and $h,h' \in K_0$.
Since $K_ m$ and $K_n$ are both normal in $K_0$ we note that
$$
K_m \cap gK_ng^{-1}=h(K_m \cap z K_n z^{-1})h^{-1}.
$$
Therefore $x \mapsto hx$ induces isomorphisms $h_*$ on cohomology which fit in the following commutative diagram, where the horizontal maps are the restriction maps:
\begin{equation*}
  \xymatrix{
  H^i(K_m \cap gK_ng^{-1},k)  \ar[r] & H^i(K_m \cap gK_{n'}g^{-1},k) \\
  H^i(K_m \cap zK_nz^{-1},k)  \ar[u]^{h_*} \ar[r] & H^i(K_m \cap zK_{n'}z^{-1},k) \ar[u]_{h_*} .}
\end{equation*}
Fix an $i>i_0$. Our claim that $\res_{n,n'}^i(g)=0$ is therefore equivalent to $\res_{n,n'}^i(z)=0$, which we now proceed to show using the decomposition in Proposition \ref{decompose}.

Consider $\res_{n,n'}^i(z)$ restricted to the piece of $H^i(K_m \cap zK_nz^{-1},k)$ indexed by $\{a_{\alpha}\}_{\alpha\in \Phi_{\red}^-}$, $b$, $\{c_{\alpha}\}_{\alpha\in \Phi_{\red}^+}$. Observe that
\begin{itemize}
\item $\bigwedge^b (Z_{n})_{\Phi}^* \longrightarrow \bigwedge^b (Z_{n'})_{\Phi}^*$ vanishes for $b>0$ since $Z_{n'}\subseteq Z_{n+e}=(Z_n)^p$ by \eqref{f:lower-series}.
\item When $\alpha \in \Phi_{\red}^+$ we have $\langle \nu(z),\alpha\rangle\leq 0$. In particular
$$
z\tilde{U}_{\alpha,n}z^{-1}=U_{\alpha,n-\langle \nu(z),\alpha\rangle}\cdot U_{2\alpha,n-\langle \nu(z),2\alpha\rangle} \subseteq \tilde{U}_{\alpha,n-\langle \nu(z),\alpha\rangle} \subseteq
\tilde{U}_{\alpha,n} \subseteq \tilde{U}_{\alpha,m},
$$
so that
$$
\tilde{U}_{\alpha,m}\cap z\tilde{U}_{\alpha,n}z^{-1}=z\tilde{U}_{\alpha,n}z^{-1}.
$$
For such $\alpha$ the map in question is
$$
\bigwedge^{c_{\alpha}} (z\tilde{U}_{\alpha,n}z^{-1})_{\Phi}^*\longrightarrow \bigwedge^{c_{\alpha}} (z\tilde{U}_{\alpha,n'}z^{-1})_{\Phi}^*
$$
which vanishes for $c_{\alpha}>0$ since
$$
z\tilde{U}_{\alpha,n'}z^{-1} \subseteq z\tilde{U}_{\alpha,n+e}z^{-1}=(z\tilde{U}_{\alpha,n}z^{-1})^p,
$$
again by \eqref{f:lower-series}.
\end{itemize}
If $\res_{n,n'}^i(z)\neq 0$ there must be a piece of cohomology with $b=0$ and all $c_{\alpha}=0$ where $\res_{n,n'}^i(z)$ is nonzero. Since
$\sum_{\alpha \in \Phi_{\red}^-} a_{\alpha}=i$ and all $a_{\alpha}\leq \dim_{\Bbb{Q}_p} U_\alpha (=\dim_{\Bbb{Q}_p} (\tilde{U}_{\alpha,m} \cap z \tilde{U}_{\alpha,n} z^{-1}))$ we conclude that
$$
i \leq \sum_{\alpha \in \Phi_{\red}^-}\dim_{\Bbb{Q}_p} U_\alpha =\dim_{\Bbb{Q}_p}(G/P_{\text{min}})=i_0.
$$
See \cite[Sect.~21.11]{Bor} for the first equality.
\end{proof}

Next we show the above bound $i_0$ is sharp, as made precise in the result below.

\begin{lemma}\label{nonvan}
Given $m\leq n\leq n'$ in $e\N$ (all $>e$ if $p=2$) there exists $z \in Z^+$ satisfying
$$
-\langle \nu(z),\alpha\rangle\leq m-n', \: \: \: \forall \alpha \in \Phi^-.
$$
For such $z$, $\res_{n,n'}^{i}(z)\neq 0$ for all $i \leq i_0$.
\end{lemma}
\begin{proof}
For the existence part, take $z:=\mu(\pi) \in S \subseteq Z$ where $\mu \in X_*({\S})^+$ is a dominant cocharacter satisfying
$\langle \mu,\alpha \rangle \leq m-n'$ for all $\alpha \in \Phi^-$. We note that $X^*({\Z})\rightarrow X^*({\S})$ has finite cokernel, so some multiple of $\alpha$
extends to an $\F$-rational character of ${\Z}$, and therefore $-\langle \nu(\mu(\pi)),\alpha\rangle=\v \alpha(\mu(\pi))=\langle \mu,\alpha \rangle$ by the defining property of $\nu$. We see that $z \in S \cap Z^+$.

For the non-vanishing part, we first work out the case $i=i_0$. Consider the contribution to $H^{i_0}(K_m \cap zK_nz^{-1},k)$ with indices $a_{\alpha}=\dim_{\Bbb{Q}_p} U_\alpha$ for all $\alpha \in \Phi_{\red}^-$ (and therefore $b=0$ and $c_{\alpha}=0$ for
$\alpha \in \Phi_{\red}^+$) which is the line
$$
\mathcal{L}_{z,n}:=\bigotimes_{\alpha \in \Phi_{\red}^-} \bigwedge^{\text{top}} (\tilde{U}_{\alpha,m}\cap z\tilde{U}_{\alpha,n}z^{-1})_{\Phi}^*.
$$
We are assuming $z \in Z^+$ satisfies the inequalities $-\langle \nu(z),\alpha\rangle\leq m-n'$ for \emph{all} $\alpha \in \Phi^-$ (even the non-reduced roots).
Under this assumption we have
$$
U_{\alpha,m}\subseteq z U_{\alpha,n'}z^{-1}\subseteq z U_{\alpha,n}z^{-1}
$$
since $m \geq n'-\langle \nu(z),\alpha\rangle$, and similarly for $2\alpha$ if it is a root. As a result
$$
\tilde{U}_{\alpha,m}\cap z\tilde{U}_{\alpha,n}z^{-1}=\tilde{U}_{\alpha,m}=\tilde{U}_{\alpha,m}\cap z\tilde{U}_{\alpha,n'}z^{-1}
$$
and consequently $\res_{n,n'}^{i_0}(z)$ maps $\mathcal{L}_{z,n}$ isomorphically to the line $\mathcal{L}_{z,n'}$ in
$H^{i_0}(K_m \cap zK_{n'}z^{-1},k)$. In particular $\res_{n,n'}^{i_0}(z)$ is nonzero on $\mathcal{L}_{z,n}$ for all such $z$.

For $i\leq i_0$ we generalize the argument from the previous paragraph as follows. Once and for all we write $i=\sum_{\alpha \in \Phi_{\red}^-}a_\alpha$ for a choice of integers $0 \leq a_\alpha \leq \dim_{\Bbb{Q}_p} U_\alpha$. One way of doing this is to list the roots, $\Phi_{\red}^-=\{\alpha_1,\alpha_2,\ldots\}$. Then let $q$ be the largest index for which
$$
\dim_{\Bbb{Q}_p} U_{\alpha_1}+\cdots+\dim_{\Bbb{Q}_p} U_{\alpha_q}\leq i.
$$
By convention $q:=0$ if $i<\dim_{\Bbb{Q}_p} \U_{\alpha_1}$. Now let $a_{\alpha_j}:=\dim_{\Bbb{Q}_p} U_{\alpha_j}$ for $j \leq q$. If $i<i_0$ we let
$$
a_{\alpha_{q+1}}:=i-(\dim_{\Bbb{Q}_p} U_{\alpha_1}+\cdots+\dim_{\Bbb{Q}_p} U_{\alpha_q}).
$$
If there are any remaining roots $\alpha_j$ with $j>q+1$ we declare that $a_{\alpha_j}:=0$.

Having chosen this expansion $i=\sum_{\alpha \in \Phi_{\red}^-}a_\alpha$, we introduce the following auxiliary subspace of $H^{i}(K_m \cap zK_nz^{-1},k)$:
$$
\mathcal{V}_{z,n}:=\bigotimes_{\alpha \in \Phi_{\red}^-} \bigwedge^{a_{\alpha}} (\tilde{U}_{\alpha,m}\cap z\tilde{U}_{\alpha,n}z^{-1})_{\Phi}^*.
$$
The restriction map $\res_{n,n'}^i(z)$ restricts to an isomorphism $\mathcal{V}_{z,n} \iso \mathcal{V}_{z,n'}$ for $z$ as above. In particular $\res_{n,n'}^{i}(z)\neq 0$ as claimed.
\end{proof}

\subsection{The proof of Theorem \ref{introrange}}

We can now prove our main result for $p$-adic reductive groups, which we recall here:

\begin{theorem}\label{main}
Fix an $m \in e\N$ ($>e$ if $p=2$) and let $i_0=\dim_{\Bbb{Q}_p}(G/P_{\text{min}})$ as above. Then:
\begin{itemize}
\item[(a)] $R^i \Ind_{K_m}^G(k)=0$ for all $i>i_0$;
\item[(b)] $R^i \Ind_{K_m}^G(k)\neq 0$ for all $i\leq i_0$.
\end{itemize}
\end{theorem}

\begin{proof}
To show the vanishing in part (a) suppose $i>i_0$. Let $n \in (e\N)_{\geq m}$ be arbitrary and consider any
$$
c \in {\varprojlim}_{g\in G/K_n} H^i(K_m \cap gK_ng^{-1}, k)
$$
as in Remark \ref{trans}. It suffices to show $t_{K_n,K_{n'}}^i(c)=0$ for all $n' \in (e\N)_{>n}$, which follows from Lemma \ref{resvan} since
$$
t_{K_n,K_{n'}}^i(c)_{gK_{n'}}=\res_{n,n'}^i(g)(c_{gK_n})=0
$$
for all $g \in G$. (We have used the formula for the transition maps $t_{K_n,K_{n'}}^i$ given in \ref{trans}.)

For the non-vanishing in part (b) now suppose $i\leq i_0$ and pick some $n \in (e\N)_{\geq m}$ once and for all. We will construct a nonzero class $c$ which survives all the transition maps. That is, such that $t_{K_n,K_{n'}}^i(c)\neq 0$ for all $n' \in (e\N)_{>n}$.

The class $c$ will be prescribed on a set of representatives $\mathcal{X}$ for $Z^+/Z^0$. (For comparison $Z^0$ is shorthand notation for the subgroup denoted by $Z(\F)^0$ in \cite[Df.~2.6.23]{KP23}.) We recall the Cartan decomposition $G=\bigsqcup_{z\in \mathcal{X}} K_0zK_0$.

For every $z \in \mathcal{X}$ we once and for all select a nonzero vector
$v_z \in \mathcal{V}_{z,n}$ (recall the notation introduced in the proof of Lemma \ref{nonvan}). Corresponding to these data there is a unique groupoid cohomology class
$$
c: G/K_n \longrightarrow {\bigoplus}_{g\in G/K_n} H^i(K_m \cap gK_ng^{-1}, k)
$$
with the following properties:
\begin{itemize}
\item $c$ is $K_m$-equivariant (in the sense described in Remark \ref{trans});
\item $c$ is supported on $K_m\mathcal{X}K_n/K_n$;
\item $c_{zK_n}=v_z$ for all $z \in \mathcal{X}$.
\end{itemize}
The uniqueness of $c$ is clear. For its existence let $g \in K_m\mathcal{X}K_n$ and write $g=hzh'$ accordingly. Thus
$h\in K_m$, $h' \in K_n$, and $z \in \mathcal{X}$. By the Cartan decomposition the factor $z$ is uniquely determined, and the left coset $h(K_m \cap zK_nz^{-1})$ is independent of the factorization of $g$. Now let $c_{gK_n}=h_*(v_z)$ which is therefore well-defined.

To see this $c$ has the desired property let $n' \in (e\N)_{>n}$. Our task is to verify the tuple $t_{K_n,K_{n'}}^i(c)$ has at least one nonzero component.
Take $z \in \mathcal{X}$ to be an element such that $-\langle \nu(z),\alpha\rangle\leq m-n'$ for all $\alpha \in \Phi^-$.
Then by (the proof of) Lemma \ref{nonvan} we know that $\res_{n,n'}^{i}(z)$ restricts to an isomorphism $\mathcal{V}_{z,n} \iso \mathcal{V}_{z,n'}$. In particular
$$
t_{K_n,K_{n'}}^i(c)_{zK_{n'}}=\res_{n,n'}^i(z)(c_{zK_n})=\res_{n,n'}^i(z)(v_z)\neq 0
$$
which finishes the proof.
\end{proof}

We finish this section by emphasizing an application to duality, which essentially reproves and strengthens one of the main results from \cite{Sor}.

\begin{corollary}\label{Sind}
Let $m \in e\N$ ($m>e$ if $p=2$). Then $\underline{\Ext}^i(\ind_{K_m}^G k,k)=0 \Longleftrightarrow i>i_0$.
\end{corollary}

\begin{proof}
Due to Corollary \ref{degen} this is just a restatement of Theorem \ref{main}.
\end{proof}

\subsection{The proof of Theorem \ref{intromain}}

By assumption $\G$ is nontrivial and connected. Hence the group $Z$ has positive $\Q_p$-dimension. To show $\Mod_k(G)$ has no nonzero projective objects,
the proof of Proposition \ref{proj} applies, noting that $R^d\Ind_{K_m}^G(k)=0$. Indeed, by  \eqref{f:Iwahori-fac-Km} for instance,
$$
d = \sum_{\alpha \in \Phi_{\red}^-}\dim_{\Bbb{Q}_p} U_\alpha+
\dim_{\Q_p} Z +
\sum_{\alpha \in \Phi_{\red}^+}\dim_{\Bbb{Q}_p} U_\alpha =
2i_0+\dim_{\Q_p} Z>i_0 \ .
$$
We immediately deduce from Theorem \ref{main} that $R^d\Ind_{K_m}^G(k)=0$.

\begin{remark}
There is a quicker route to Theorem \ref{intromain} by directly showing $R^d\Ind_{K_m}^G(k)=0$. This amounts to verifying the condition appearing in the proof of Proposition \ref{proj}. Indeed, for all large enough $n$ there is an $n'>n$ (for instance $n'=n+1$) such that we have a \emph{strict} inclusion
$$
K_m \cap gK_{n'}g^{-1} \subsetneq K_m \cap gK_{n}g^{-1}
$$
for all $g \in G$. Indeed, by the Cartan decomposition it suffices to check this for $g=z \in Z^+$. For such $z$ the strict inclusion follows from the factorization (\ref{iwa}) by noting that
$Z_{n'} \subsetneq Z_n$.
\end{remark}

\section{The derived functor $R\underline{\Ind}$}

\subsection{Preliminary remarks}

We return to the general setup and let $G$ be an arbitrary $p$-adic Lie group of $\Bbb{Q}_p$-dimension $d$, and $k$ is still a field of characteristic $p$. Recall that $\Mod_k(G)$ is the abelian category of smooth $G$-representations on $k$-vector spaces. We let $D(G)$ denote its (unbounded) derived category.

We remind the reader that our convention is that $G$ acts on $\Ind_K^G(V)$ by right translations. Thus, for any $V$ in $\Mod_k(K)$, the space $\Ind_K^G(V)$ consists of all
functions $f:G \longrightarrow V$ such that

\begin{itemize}
  \item $f(\kappa g) = \kappa f(g)$ for any $g \in G$ and $\kappa \in K$;
  \item $f(gu) = f(g)$ for any $g \in G$ and $u \in U$ for some open $U \leq G$ (depending on $f$).
\end{itemize}

The fact that $\Ind_K^G$ is right adjoint to restriction is a consequence of Frobenius reciprocity (see \cite[I.5.7(i)]{Vig96} for example) which is the isomorphism
\begin{align}\label{f:adj-K}
  \Hom_{\Mod_k(K)}(W,V) & \xrightarrow{\;\cong\;} \Hom_{\Mod_k(G)}(W, \Ind_K^G(V)) \\
     A & \longmapsto B(v)(g) := A(gv)     \nonumber
\end{align}
for $V$ in $\Mod_k(K)$ and $W$ in $\Mod_k(G)$. The functor $\Ind_K^G$ is left exact and preserves injective objects (see \cite[I.5.9]{Vig96}). If $K' \leq K$ then $\Ind_K^G(V) \subseteq \Ind_{K'}^G(V)$. We may therefore take the union over $K$ and introduce the left exact functor
\begin{align*}
  \underline{\Ind} : \Mod_k(G) & \longrightarrow \Mod_k(G) \\
                          V & \longmapsto \varinjlim_K \Ind_K^G(V) \ .
\end{align*}
Later we will endow $\underline{\Ind}(V)$ with a smooth $G \times G$-action, provided $V$ is in $\Mod(G)$. We are interested in the derived functors $R^i \underline{\Ind}$ and their relation to $R^i\Ind_K^G$.

\begin{lemma}\label{lemma:Ind=lim}
  $R^i\underline{\Ind}(V) = \varinjlim_K R^i\Ind_K^G(V)$.
\end{lemma}
\begin{proof}
This is immediate from the exactness of inductive limits and the fact that any injective object in $\Mod_k(G)$ remains injective in any $\Mod_k(K)$.
(Indeed, the left adjoint of restriction is compact induction, which is exact since $k[G]$ is flat over $k[K]$.)
\end{proof}

As a preliminary observation, we prove that $R^i \underline{\Ind}$ vanishes for $i>d$:

\begin{proposition}\label{prop:Ind-fcd}
The following holds:
\begin{itemize}
  \item[i.] For any compact open subgroup $K \subseteq G$, the functor $\Ind_K^G$ has cohomological dimension at most $d$.
  \item[ii.] The functor $\underline{\Ind}$ has cohomological dimension at most $d$.
\end{itemize}
\end{proposition}

\begin{proof}
This is a consequence of Proposition \ref{RInd}, but we prefer to include the following direct argument (which also gives an alternate proof of the derived Mackey decomposition \ref{RInd}).

i. Consider any $V$ in $\Mod_k(K)$ and let $U \subseteq G$ be a compact open subgroup. We choose a set $R \subseteq G$ of representatives for the double cosets
$K \backslash G/U$. The Mackey decomposition (see \cite[I.5.5]{Vig96} for example) is a natural isomorphism
\begin{align*}
  \Ind_K^G(V)^{U} & \xrightarrow{\;\cong\;} \prod_{x \in R} V^{K \cap xUx^{-1}}   \\
        f & \longmapsto (f(x))_{x\in R} \ .
\end{align*}
Let $V \overset{\text{qis}}{\longrightarrow} \mathcal{J}^\bullet$ be a choice of injective resolution in $\Mod_k(K)$. On the one hand it remains an injective resolution in any $\Mod_k(K \cap xUx^{-1})$. Hence the complex $\prod_{x \in R} (\mathcal{J}^\bullet)^{K \cap xUx^{-1}}$ computes $\prod_{x \in R} H^*(K \cap xUx^{-1},V)$. On the other hand the complex $\Ind_K^G(\mathcal{J}^\bullet)$ computes $R^*\Ind_K^G(V)$ and is a complex of injective objects in $\Mod_k(G)$. Therefore the composed functor spectral sequence for the functors $(-)^{U}$ and $\Ind_K^G(-)$ exists and reads
\begin{equation*}
  E_2^{r,s} = H^r(U, R^s\Ind_K^G(V)) \Longrightarrow \prod_{x \in R} H^{r+s}(K \cap xUx^{-1},V) \ .
\end{equation*}
By passing to the limit with respect to $U$ it degenerates into isomorphisms
\begin{equation}\label{f:derivedMackey}
  R^s\Ind_K^G(V) \cong \varinjlim_{U} \prod_{x \in R} H^s(K \cap xUx^{-1},V) \ .
\end{equation}
For this note that profinite group cohomology commutes with inductive limits and that, for any $M$ in $\Mod_k(K)$, we have
\begin{equation*}
  \varinjlim_{U} H^r(U,M) =
  \begin{cases}
  M & \text{if $r=0$},  \\
  0 & \text{otherwise}.
  \end{cases}
\end{equation*}
In the limit (\ref{f:derivedMackey}) we may take $U$ to run over Poincar\'e subgroups of $G$. However, with $U$ the open subgroup $K \cap xUx^{-1}$ is also a Poincar\'e group of dimension $d$. Therefore all the cohomology groups on the right-hand side of \eqref{f:derivedMackey} vanish for $s>d$.

ii. Because of Lemma \ref{lemma:Ind=lim} this follows from i.
\end{proof}

\begin{remark}\label{rem:Ind-exact}
   If $G$ is compact then the functors $\Ind_K^G$ and $\underline{\Ind}$ are exact.
\end{remark}
\begin{proof}
For compact $G$ we have $\Ind_K^G = \ind_K^G=\text{compact induction}$, which is exact.
\end{proof}

As a consequence of Prop.\ \ref{prop:Ind-fcd}.ii, $\underline{\Ind}$ has finite cohomological dimension, and we therefore have (by \cite[Cor.~I.5.3($\gamma$)]{Har}) the total derived functor between the unbounded derived categories:
\begin{equation*}
  R\underline{\Ind} : D(G) \longrightarrow D(G) \ .
\end{equation*}
This functor has more structure, as we now explain. In the following we use the convention that for a $G \times G$-action we write $G_\ell$, resp.\ $G_r$, if we refer to the action of the left, resp.\ right, factor in the product $G \times G$.

\begin{lemma}\label{lemma:GG}
  For any two representations $V$ and $V'$ in $\Mod_k(G)$ we have
\begin{itemize}
  \item[i.] For $f \in \Ind_K^G(V')$ and $(g_1,g_2) \in G \times G$ the function
  $$
  {^{(g_1,g_2)} f}(g) := g_1 f(g_1^{-1}gg_2)
  $$
  lies in $\Ind_{g_1 K g_1^{-1}}^G(V')$;
  \item[ii.] this defines a smooth $G \times G$-action on $\underline{\Ind}(V')$; more precisely, $\Ind_K^G(V') = \underline{\Ind}(V')^{K_\ell}$;
  \item[iii.] the earlier $G$-action on $\underline{\Ind}(V')$ is the $G_r$-action;
  \item[iv.] the adjunction isomorphism
\begin{equation}\label{f:adj}
     \underline{\Hom}(V,V') \xrightarrow{\;\cong\;} \underline{\Hom}_{\Mod_k(G_r)}(V,\underline{\Ind}(V')):=\varinjlim_K  \Hom_{\Mod_k(G_r)}(V,\underline{\Ind}(V'))^{K_\ell}
\end{equation}
   obtained by passing to the inductive limit with respect to $K$ in \eqref{f:adj-K} is $G$-equivariant where $G$ acts on the target through the $G_\ell$-action on $\underline{\Ind}(V')$;
  \item[v.] if the representation $V$ is finitely generated then
\begin{equation*}
  \underline{\Hom}_{\Mod_k(G_r)}(V,\underline{\Ind}(V')) = \Hom_{\Mod_k(G_r)}(V,\underline{\Ind}(V')) \ .
\end{equation*}
\end{itemize}
\end{lemma}
\begin{proof}
i. Suppose $f$ is fixed by right translation by $U$. Then ${^{(g_1,g_2)} f}$ is fixed by right translation by $g_2 U g_2^{-1}$. Furthermore, for $\kappa \in K$, we compute
\begin{equation*}
  {^{(g_1,g_2)} f}((g_1\kappa g_1^{-1})g) = g_1 f(g_1^{-1} (g_1\kappa g_1^{-1})g g_2) =g_1\kappa f(g_1^{-1}g g_2)=(g_1\kappa g_1^{-1})\cdot {^{(g_1,g_2)} f}(g) \ .
\end{equation*}
Thus ${^{(g_1,g_2)} f} \in \underline{\Ind}(V')^{g_1Kg_1^{-1} \times g_2 U g_2^{-1}}$. Now ii and iii are obvious.

iv. For a fixed $K$ we have by \eqref{f:adj-K} and ii:
\begin{align*}
  \Hom_{\Mod_k(K)}(V,V') & \xrightarrow{\cong} \Hom_{\Mod_k(G)}(V,\Ind_K^G(V'))  \\
   & = \Hom_{\Mod_k(G)}(V,\underline{\Ind}(V')^{K_\ell })      \\
   & = \Hom_{\Mod_k(G_r)}(V,\underline{\Ind}(V'))^{K_\ell}
\end{align*}
which in the limit with respect to $K$ gives rise to the isomorphism \eqref{f:adj}. A straightforward computation shows its equivariance. (We refer to \cite[p.~32]{SS} for the definition of the smooth linear maps $\underline{\Hom}(V,V')$.)

v. Under the finiteness assumption, the image of any $G$-homomorphism from $V$ into $\underline{\Ind}(V')$ lies in $\Ind_K^G(V')$ for some open $K \leq G$ (which depends on the homomorphism).
\end{proof}

Hence we actually have a left exact functor $\underline{\Ind} : \Mod_k(G) \longrightarrow \Mod_k(G \times G)$ which derives to a functor  $R\underline{\Ind} : D(G) \longrightarrow D(G \times G)$ computable by homotopically injective resolutions. Our next goal is to lift the adjunction \eqref{f:adj} to the derived level, using \cite[Thm.~14.4.8]{KS}.
For this we first have to discuss its right-hand side in more detail.

We begin by recalling some general nonsense about the adjunction between tensor products and the $\Hom$-functor which for three $k$-vector spaces $V_1$, $V_2$, and $V_3$ is given by the linear isomorphism
\begin{align}\label{f:basic-adj}
  \Hom_k(V_1 \otimes_k V_2, V_3) & \xrightarrow{\;\cong\;} \Hom_k(V_1, \Hom_k(V_2,V_3)) \\
                  A & \longmapsto \lambda_A(v_1)(v_2) := A (v_1 \otimes v_2) \ .   \nonumber
\end{align}
Suppose that all three vector spaces carry a left $G$-action. Then $\Hom_k(V_1 \otimes_k V_2, V_3)$ and $\Hom_k(V_1, \Hom_k(V_2,V_3))$ are equipped with the $G \times G \times G$-action defined by
\begin{equation*}
  {^{(g_1,g_2,g_3)} A}(v_1 \otimes v_2) := g_3 A(g_1^{-1} v_1 \otimes g_2^{-1} v_2)  \quad\text{and}\quad   {^{(g_1,g_2,g_3)} \lambda}(v_1)(v_2) := g_3 (\lambda(g_1^{-1} v_1)(g_2^{-1} v_2))  ,
\end{equation*}
respectively. The above adjunction is equivariant for these two actions. If we restrict to the diagonal $G$-action, then the above adjunction induces the adjunction isomorphism
\begin{equation*}
  \Hom_{k[G]}(V_1 \otimes_k V_2, V_3) \xrightarrow{\;\cong\;} \Hom_{k[G]}(V_1, \Hom_k(V_2,V_3)) \ .
\end{equation*}
If the $G$-action on the $V_i$ is smooth then this can also be written as an isomorphism
\begin{equation}\label{f:smooth-adj}
  \Hom_{\Mod_k(G)}(V_1 \otimes_k V_2, V_3) \cong \Hom_{\Mod_k(G)}(V_1, \underline{\Hom}(V_2,V_3)) \ .
\end{equation}
In the next paragraph we discuss a variant of this.

Now suppose that in the adjunction \eqref{f:basic-adj} the vector spaces $V_1$ and $V_2$ carry a $G$-action whereas $V_3$ carries a $G \times G$-action. Then $V_1 \otimes_k V_2$ carries a $G \times G$-action as well. Moreover, $\Hom_{k[G_r]}(V_2,V_3)$ still carries a $G$-action through the $G_\ell$-action on $V_3$. The above adjunction induces the adjunction isomorphism
\begin{equation*}
  \Hom_{k[G \times G]}(V_1 \otimes_k V_2, V_3) \xrightarrow{\;\cong\;} \Hom_{k[G_\ell]}(V_1, \Hom_{k[G_r]}(V_2,V_3)) \ .
\end{equation*}
Suppose in addition that the actions on the $V_i$ are smooth. Then the $G_\ell$-action on
\begin{equation*}
  \underline{\Hom}_{\Mod_k(G_r)}(V_2,V_3) := \varinjlim_{K} \Hom_{\Mod_k(G_r)}(V_2,V_3)^{K_\ell} \ ,
\end{equation*}
where $K$ runs over the compact open subgroups, is smooth. Hence we may rewrite the latter isomorphism as
\begin{equation*}
  \Hom_{\Mod_k(G \times G)}(V_1 \otimes_k V_2, V_3) \cong \Hom_{\Mod_k(G_\ell)}(V_1, \underline{\Hom}_{\Mod_k(G_r)}(V_2,V_3)) \ .
\end{equation*}
This works as well with $V_1$ and $V_2$ interchanged:

\begin{remark}
We could have defined the initial adjunction for vector spaces analogously by the linear isomorphism
\begin{align*}
  \Hom_k(V_1 \otimes_k V_2, V_3) & \xrightarrow{\;\cong\;} \Hom_k(V_2, \Hom_k(V_1,V_3)) \\
                  A & \longmapsto \mu_A(v_2)(v_1) := A (v_1 \otimes v_2) \ .
\end{align*}
This leads to the isomorphism
\begin{equation*}
  \Hom_{\Mod_k(G \times G)}(V_1 \otimes_k V_2, V_3) \cong \Hom_{\Mod_k(G_r)}(V_2, \underline{\Hom}_{\Mod_k(G_\ell)}(V_1,V_3)) \ .
\end{equation*}
\end{remark}

We are now in a position to apply \cite[Thm.~14.4.8]{KS}. For our present context we conclude that the functor
\begin{align*}
  \Mod_k(G)^{op} \times \Mod_k(G \times G) & \longrightarrow \Mod_k(G)  \\
  (V,V') & \longmapsto \underline{\Hom}_{\Mod_k(G_r)}(V,V')
\end{align*}
has the left adjoint functor
\begin{align*}
  \Mod_k(G) \times \Mod_k(G) & \longrightarrow \Mod_k(G \times G)  \\
  (V,V') & \longmapsto V \otimes_k V'
\end{align*}
This shows that the adjointness requirements in \cite[Thm.~14.4.8]{KS} are satisfied, so that we have the total derived functor
\begin{equation*}
  R\underline{\Hom}_{\Mod_k(G_r)} : D(G)^{op} \times D(G \times G) \longrightarrow D(G)
\end{equation*}
satisfying (14.4.6) in \cite{KS}. Namely, for $V_1^\bullet$, $V_2^\bullet$ in $D(G)$ and $V_3^\bullet$ in $D(G\times G)$ we have isomorphisms
\begin{align*}
R\Hom_{\Mod_k(G \times G)}(V_1^\bullet \otimes_k V_2^\bullet, V_3^\bullet) & \cong R\Hom_{\Mod_k(G_\ell)}(V_1^\bullet, R\underline{\Hom}_{\Mod_k(G_r)}(V_2^\bullet,V_3^\bullet)) \\
& \cong R\Hom_{\Mod_k(G_\ell)}(V_2^\bullet, R\underline{\Hom}_{\Mod_k(G_r)}(V_1^\bullet,V_3^\bullet)).
\end{align*}
The derived functor is computable via homotopically injective resolutions, see part (ii) of \cite[Thm.~14.4.8]{KS}:
\begin{equation}\label{f:computation2}
  R\underline{\Hom}_{\Mod_k(G_r)}(V^\bullet,V'^\bullet) = \mathrm{tot}_\Pi \underline{\Hom}_{\Mod_k(G_r)}(V^\bullet,J'^\bullet)
\end{equation}
where $V'^\bullet \xrightarrow{\simeq} J'^\bullet$ is a homotopically injective resolution in $\Mod_k(G \times G)$. Here $ \mathrm{tot}_\Pi$ is the direct product totalization of the double complex. We emphasize that $\Pi$ refers to the direct product in $\Mod_k(G)$, in other words the smooth vectors of the direct product of abstract $k[G]$-modules.


\subsection{Duality on compact objects}

Next we relate the duality functor $R\underline{\Hom}(-,k)$ from \cite{SS} to the object $R\underline{\Ind}(k)$. In this section we are primarily interested in the restriction of the duality functor to $D(G)^c$ (the subcategory of compact objects). Our main result here is Corollary \ref{coro:tau-compact} below.

First, let $V_1^\bullet$ and $V_2^\bullet$ be any two objects of $D(G)$ and fix a homotopically injective resolution $V_2^\bullet \xrightarrow{\simeq} J^\bullet$. Then, by
\cite[Prop.~3.1] {SS}:
\begin{equation}\label{f:computation}
  R\underline{\Hom}(V_1^\bullet,V_2^\bullet) = \underline{\Hom}^\bullet(V_1^\bullet,J^\bullet) \quad \text{and} \quad R\underline{\Ind}(V_2^\bullet) = \underline{\Ind}(J^\bullet) \ .
\end{equation}
By Lemma \ref{lemma:GG}.iv. we have the $G$-equivariant adjunction isomorphism of actual complexes
\begin{equation}\label{f:derived-adj}
  \underline{\Hom}^\bullet(V_1^\bullet,J^\bullet) \xrightarrow{\cong} \mathrm{tot}_\Pi \underline{\Hom}_{\Mod_k(G_r)} (V_1^\bullet, \underline{\Ind}(J^\bullet)) \ .
\end{equation}
Using \eqref{f:computation}, \eqref{f:derived-adj}, and \eqref{f:computation2} we now consider the $G$-equivariant map
\begin{align*}
 \tau_{V_1^\bullet,V_2^\bullet} : R\underline{\Hom}(V_1^\bullet,V_2^\bullet) &= \underline{\Hom}^\bullet(V_1^\bullet,J^\bullet) \\
 & \cong \mathrm{tot}_\Pi \underline{\Hom}_{\Mod_k(G_r)} (V_1^\bullet, \underline{\Ind}(J^\bullet)) \\
 & \rightarrow \mathrm{tot}_\Pi \underline{\Hom}_{\Mod_k(G_r)}(V_1^\bullet, J_{\underline{\Ind}(J^\bullet)}^\bullet) \\
 &= R\underline{\Hom}_{\Mod_k(G_r)}(V_1^\bullet,\underline{\Ind}(J^\bullet)) \\
 &=R\underline{\Hom}_{\Mod_k(G_r)}(V_1^\bullet,R\underline{\Ind}(V_2^\bullet))
\end{align*}
where $\underline{\Ind}(J^\bullet) \xrightarrow{\simeq} J_{\underline{\Ind}(J^\bullet)}^\bullet$ is a homotopically injective resolution in $\Mod_k(G \times G)$.

\begin{proposition}\label{prop:tau-X}
 Let $U \subseteq G$ be a torsionfree pro-$p$ open subgroup. Viewing $\mathbf{X}_U:=\ind_U^G(k)$ as a complex (concentrated in degree zero) the above map $\tau_{\mathbf{X}_U,V_2^\bullet}$ is a quasi-isomorphism for any $V_2^\bullet$.
\end{proposition}
\begin{proof}
We have to show that the map
\begin{equation*}
  \underline{\Hom}_{\Mod_k(G_r)} (\mathbf{X}_U, \underline{\Ind}(J^\bullet))  \longrightarrow
 \underline{\Hom}_{\Mod_k(G_r)}(\mathbf{X}_U, J_{\underline{\Ind}(J^\bullet)}^\bullet)
\end{equation*}
is a quasi-isomorphism. Frobenius reciprocity implies that
\begin{equation}\label{f:Frobenius}
  \underline{\Hom}_{\Mod_k(G_r)}(\mathbf{X}_U,-) = \varinjlim_{K} \Hom_{\Mod_k(G_r)}(\mathbf{X}_U,-)^{K \times \{1\}} = \varinjlim_{K} (-)^{K \times U} = (-)^{\{1\} \times U} \ .
\end{equation}
Hence the map in question is the map $\underline{\Ind}(J^\bullet)^{\{1\} \times U} \longrightarrow (J_{\underline{\Ind}(J^\bullet)}^\bullet)^{\{1\} \times U}$. Of course we have the isomorphism $H^*(U,\underline{\Ind}(J^\bullet)) \xrightarrow{\cong} H^*(U, J_{\underline{\Ind}(J^\bullet)}^\bullet)$ where, for simplicity, we write simply $U$ instead of $\{1\} \times U$.
We adopt this convention for the rest of this proof. Hence it is enough to verify the following:
\begin{itemize}
  \item[(a)] $H^*(U,\underline{\Ind}(J^\bullet)) = h^*(\underline{\Ind}(J^\bullet)^U)$;
  \item[(b)] $H^*(U, J_{\underline{\Ind}(J^\bullet)}^\bullet) = h^*((J_{\underline{\Ind}(J^\bullet)}^\bullet)^U)$.
\end{itemize}
We first note (\cite[1.3.5]{Lur}, \cite[3.1]{ScSc}) that we may assume that all our homotopically injective resolutions are even semi-injective, i.e., in addition each of their terms is an injective object.

\medskip

$\bullet$ \emph{Part (a)}

\medskip

Each $J^m$, for $m \in \mathbb{Z}$, is injective in $\Mod_k(G)$ and hence in $\Mod_k(K)$. Then any $\Ind_K^G(J^m)$ is injective in $\Mod_k(G)$ and hence in $\Mod_k(U)$. Since the cohomology of $U$ commutes with inductive limits it follows that the complex $\underline{\Ind}(J^\bullet)$ consists of $H^0(U,-)$-acyclic objects. The functor $H^0(U,-)$ has finite cohomological dimension. Therefore a) holds by \cite[Cor.~I.5.3($\gamma$)]{Har} (and its proof).

\medskip

$\bullet$ \emph{Part (b)}

\medskip

This is a similar argument since in the subsequent Lemma \ref{lemma:acyclic} we will show that any term of the complex $J_{\underline{\Ind}(J^\bullet)}^\bullet$ being injective in $\Mod_k(G \times G)$ is $H^0(U,-)$-acyclic.
\end{proof}

At the end of the previous proof we alluded to:

\begin{lemma}\label{lemma:acyclic}
   Let $U \subseteq G$ be an open subgroup. Then any injective object $V$ in $\Mod_k(G \times G)$ is $H^0(\{1\} \times U,-)$-acyclic.
\end{lemma}
\begin{proof}
First we allow an arbitrary object $V$ in $\Mod_k(G \times G)$. For any compact open subgroup $K \subset G$ we have the Hochschild-Serre spectral sequence
(where we write $U$ instead of $\{1\}\times U$ on the $E_2$-page):
\begin{equation*}
  E_2^{rs} = H^r(K, H^s(U,V)) \Longrightarrow H^{r+s}(K \times U,V) \ .
\end{equation*}
It is functorial with respect to the restriction to a smaller compact open subgroup $K' \subseteq K$ (see \cite[II.4 Ex.~3]{NSW}). Hence we may pass to the limit with respect to smaller $K' \subseteq K$ in this spectral sequence. As in the proof of Prop.\ \ref{prop:Ind-fcd}.i the limit spectral sequence degenerates into isomorphisms
\begin{equation*}
  H^s(U,V) \cong \varinjlim_{K'} H^s(K' \times U,V) \ .
\end{equation*}
Now, if $V$ is injective in $\Mod_k(G \times G)$ then it is injective in each $\Mod_k(K' \times U)$ so that the above right-hand side vanishes for $s > 0$. Therefore so does
$H^s(U,V)$.
\end{proof}

We easily deduce the following result.

\begin{corollary}\label{coro:nat-isos-X}
  Let $U \subseteq G$ be a torsionfree pro-$p$ open subgroup; we then have:
\begin{itemize}
  \item[i.] The functors $R\underline{\Hom}_{\Mod_k(G_r)}(\mathbf{X}_U,-)$ and $RH^0(\{1\} \times U,-)$ from $D(G \times G)$ to $D(G)$ are naturally isomorphic;
  \item[ii.] the functors $R\underline{\Hom}(\mathbf{X}_U,-)$ and $RH^0(\{1\} \times U,R\underline{\Ind}(-))$ from $D(G)$ to $D(G)$ are naturally isomorphic;
  \item[iii.] if $G$ is compact, then the functors $R\underline{\Hom}(\mathbf{X}_U,-)$ and $RH^0(\{1\} \times U,\underline{\Ind}(-))$ from $D(G)$ to $D(G)$ are naturally isomorphic.
\end{itemize}
\end{corollary}
\begin{proof}
i. This follows from \eqref{f:Frobenius} and the above Lemma \ref{lemma:acyclic}. ii. Combine i. and Prop.\ \ref{prop:tau-X}. iii. If $G$ is compact then, by Remark \ref{rem:Ind-exact}, the functor $\underline{\Ind}$ is exact so that $R\underline{\Ind} = \underline{\Ind}$. Now apply part ii.
\end{proof}

We now specialize the map $\tau_{V_1^\bullet,V_2^\bullet}$ to the case where $V_2^\bullet$ is the complex with the trivial representation $k$ in degree zero, and obtain a natural transformation
\begin{equation}\label{f:tau}
  \tau_{-} := \tau_{-,k} : R\underline{\Hom}(-,k) \longrightarrow R\underline{\Hom}_{\Mod_k(G_r)}(-,R\underline{\Ind}(k))
\end{equation}
between exact functors from $D(G)$ to $D(G)$.

Recall from \cite[Rk.~10]{DGA} that the full subcategory $D(G)^c$ of all compact objects in $D(G)$ is the smallest strictly full triangulated subcategory closed under direct summands which contains $\mathbf{X}_U$ for some (or equivalently any) open torsionfree pro-$p$ subgroup $U \subseteq G$.

\begin{corollary}\label{coro:tau-compact}
  $\tau_{-}$ restricted to $D(G)^c$ is a natural isomorphism.
\end{corollary}
\begin{proof}
The full subcategory of all objects $V^\bullet$ in $D(G)$ for which $\tau_{V^\bullet}$ is an isomorphism is a strictly full triangulated subcategory closed under direct summands which contains $\mathbf{X}_U$ by Prop.\ \ref{prop:tau-X}.
\end{proof}


\subsection{The complex $R\underline{\Ind}(k)$ for reductive groups}\label{sec:RIndk}

In this section we again focus on $p$-adic reductive groups, and we put ourselves in the setup from Section \ref{red}.
Thus $G=\mathbf{G}(\mathfrak{F})$ is the group of $\F$-rational points of a connected reductive group $\mathbf{G}$ defined over some finite extension $\mathfrak{F}/\mathbb{Q}_p$.
The goal in this section is to establish the following vanishing result.

\begin{theorem}\label{thm:vanishing}
  $R^i\underline{\Ind}(k) = 0$ for all $i>0$.
\end{theorem}

The proof requires some preparation. We start with the following observation.

\begin{lemma}\label{lem:res}
  Let $U' \subset U$ be two uniform pro-$p$ groups; if $U' \subset U^p$, then the restriction map
  $$
  H^i(U,k) \overset{\text{res}}{\longrightarrow} H^i(U',k)
  $$
  is the zero map for all $i>0$.
\end{lemma}
\begin{proof}
For $i = 1$ the map in question is the natural map $\Hom_k(U/U^p,k) \rightarrow \Hom_k(U'/U'^p,k)$ which is the zero map by assumption. For uniform pro-$p$ groups, and a coefficient field $k$ of characteristic $p$, the cohomology is generated under the cup product in degree $1$ (one can reduce to the case $k=\Bbb{F}_p$ which is \cite[Prop.~2.5.7.1, p.~567]{Laz}). Since restriction maps commute with cup products the assertion follows.
\end{proof}

To apply Lemma \ref{lem:res} the key input is the following.

\begin{lemma}\label{lem:ppower}
Let $n \in e\N$ (and assume $n>e$ if $p=2$). Then, for all $g \in G$ we have
$$
(K_n \cap gK_ng^{-1})^p=K_{n+e} \cap gK_{n+e}g^{-1} .
$$
\end{lemma}

\begin{proof}
By the Cartan decomposition $G=K_0Z^+K_0$ we may assume that $g=z \in Z^+$, noting that $K_n$ and $K_{n+e}$ are both normal subgroups of $K_0$. For any $z \in Z$ we have,
by \eqref{iwa} in the proof of Proposition \ref{decompose}, the homeomorphism
\begin{equation*}
\big(\prod_{\alpha \in \Phi_{\red}^-} \tilde{U}_{\alpha,n}\cap z\tilde{U}_{\alpha,n}z^{-1}\big)  \times Z_n \times \big(\prod_{\alpha \in \Phi_{\red}^+} \tilde{U}_{\alpha,n}\cap z\tilde{U}_{\alpha,n}z^{-1} \big) \iso K_n \cap zK_nz^{-1}.
\end{equation*}
In that proof we also argued that by dividing out by the subgroups of $p$-powers (even if $p=2$) we get an isomorphism. Hence the above map restricted to $p$-powers must still be a homeomorphism, i.e.,
\begin{equation*}
  (K_n \cap zK_nz^{-1})^p = \prod_{\alpha \in \Phi_{\text{red}}^-} ( \tilde{U}_{\alpha,n}\cap z\tilde{U}_{\alpha,n}z^{-1})^p \times Z_n^p \times  \prod_{\alpha \in \Phi_{\text{red}}^+} (\tilde{U}_{\alpha,n}\cap z\tilde{U}_{\alpha,n}z^{-1})^p \ .
\end{equation*}
Using \eqref{f:ZconjU} we compute, now for $z \in Z^+$,
\begin{align*}
   z\tilde{U}_{\alpha,n}z^{-1}
   & = z U_{\alpha,n}z^{-1} \cdot z U_{2\alpha,n}z^{-1} \\
   &  = U_{\alpha,n - \langle \nu(z),\alpha \rangle} \cdot U_{2\alpha,n - \langle \nu(z),2\alpha \rangle}  \\
   &
   \begin{cases}
   \supseteq U_{\alpha,n} \cdot U_{2\alpha,n} = \tilde{U}_{\alpha,n} & \text{if $\alpha \in \Phi^-$},  \\
   \subseteq U_{\alpha,n} \cdot U_{2\alpha,n} = \tilde{U}_{\alpha,n} & \text{if $\alpha \in \Phi^+$}
   \end{cases}
\end{align*}
and therefore
\begin{equation*}
  \tilde{U}_{\alpha,n}\cap z \tilde{U}_{\alpha,n}z^{-1} =
  \begin{cases}
   \tilde{U}_{\alpha,n}  & \text{if $\alpha \in \Phi^-$},  \\
   z \tilde{U}_{\alpha,n} z^{-1} & \text{if $\alpha \in \Phi^+$}.
   \end{cases}
\end{equation*}
Using \eqref{f:lower-series} it follows that
\begin{equation*}
  (\tilde{U}_{\alpha,n}\cap z\tilde{U}_{\alpha,n}z^{-1})^p = \tilde{U}_{\alpha,n}^p \cap z\tilde{U}_{\alpha,n}^p z^{-1} = \tilde{U}_{\alpha,n+e}\cap z\tilde{U}_{\alpha,n+e}z^{-1}
\end{equation*}
as well as $Z_n^p = Z_{n+e}$. We conclude that
\begin{align*}
(K_n \cap zK_nz^{-1})^p & =
\prod_{\alpha \in \Phi_{\text{red}}^-} (\tilde{U}_{\alpha,n+e} \cap z\tilde{U}_{\alpha,n+e}z^{-1}) \times Z_{n+e} \times \prod_{\alpha \in \Phi_{\text{red}}^+} (\tilde{U}_{\alpha,n+e} \cap z\tilde{U}_{\alpha,n+e}z^{-1})\\
&=K_{n+e} \cap zK_{n+e}z^{-1}
\end{align*}
as desired.
\end{proof}

We can now prove our vanishing result for $R^i\underline{\Ind}(k)$.

\begin{proof} \emph{(Theorem \ref{thm:vanishing}.)}
Upon passing to the diagonal colimit $K=U$ when combining Proposition \ref{RInd} and Lemma \ref{lemma:Ind=lim} we see that
\begin{align*}
R^i\underline{\Ind}(k) & \simeq {\varinjlim}_K {\varinjlim}_U {\varprojlim}_{x\in G/U} H^i(K \cap xUx^{-1}, k) \\
& \simeq {\varinjlim}_U {\varprojlim}_{x\in G/U} H^i(U \cap xUx^{-1}, k) \\
& \simeq {\varinjlim}_{n\in e\N} {\varprojlim}_{x\in G/K_n} H^i(K_n \cap xK_nx^{-1}, k).
\end{align*}
Now it is obvious from Lemmas \ref{lem:res} and \ref{lem:ppower} that the transition maps
\begin{align*}
{\varprojlim}_{x\in G/K_n} H^i(K_n \cap xK_nx^{-1}, k) & \longrightarrow {\varprojlim}_{x'\in G/K_{n'}} H^i(K_{n'} \cap x'K_{n'}x'^{-1}, k) \\
c & \longmapsto \big( \res_{K_{n'} \cap x'K_{n'}x'^{-1}}^{K_n \cap x'K_nx'^{-1}}c_{x'K_n} \big)_{x' \in G/K_{n'}}
\end{align*}
are all zero for $i>0$ and $n'=n+e$. Therefore $R^i\underline{\Ind}(k)=0$ for $i>0$.
\end{proof}

Since the complex $R\underline{\Ind}(k)$ is concentrated in non-negative degrees, and has zero cohomology in positive degrees by \ref{thm:vanishing}, there is a quasi-isomorphism
\begin{equation}\label{f:Ind-RInd}
  \underline{\Ind}(k)[0] \overset{\text{qis}}{\longrightarrow}R\underline{\Ind}(k)
\end{equation}
for $p$-adic reductive groups $G$. Note that $\underline{\Ind}(k)$ is simply the space $\mathcal{C}^\infty(G,k)$ of smooth vectors in the $G \times G$-representation on the space of all $k$-valued functions on $G$.

\begin{corollary}
   The functors $R\underline{\Hom}(-,k)$ and $R\underline{\Hom}_{\Mod_k(G_r)}(-, \mathcal{C}^\infty(G,k))$ restricted to $D(G)^c$ are isomorphic.
\end{corollary}
\begin{proof}
Combine Corollary \ref{coro:tau-compact} with \eqref{f:Ind-RInd}.
\end{proof}



\bigskip

\noindent {\it{E-mail addresses}}: {\texttt{pschnei@uni-muenster.de}, {\texttt{csorensen@ucsd.edu}}

\bigskip

\noindent {\sc{Peter Schneider, Math. Institut, Universit\"a{}t M\"{u{nster, M\"{u}nster, Germany.}}

\bigskip

\noindent {\sc{Claus Sorensen, Dept. of Mathematics, UC San Diego, La Jolla, USA.}}

\end{document}